\documentclass[12pt]{amsart}
\usepackage{graphicx} %
\usepackage{amsfonts}
\usepackage{amssymb}
\usepackage{amsmath}
\usepackage{amsthm}
\usepackage{tikz-cd}
\usepackage{geometry}
\usepackage{nicematrix}
\usepackage{subcaption}
\usepackage{comment}
\usepackage{url}
\usepackage{hyperref}
\theoremstyle{plain}
\newtheorem{theorem}{Theorem}[section]
\newtheorem{definition}[theorem]{Definition}
\newtheorem{example}[theorem]{Example}
\newtheorem{remark}[theorem]{Remark}

\newtheorem{corollary}[theorem]{Corollary}
\newtheorem{prop}[theorem]{Proposition}

\title{Positive Steady-state Varieties of Small Chemical Reaction Networks}
\author[M. Curiel]{Maize Curiel}
\address{Department of Mathematics, University of Hawai`i at M\=anoa, 2565 McCarthy Mall (Keller Hall 401A)
Honolulu, Hawaii 96822, United States}
\email{curielm@hawaii.edu}
\author[E. Farr]{Elise Farr}
\address{Department of Mathematics and Statistics,  Boston University, 665 Commonwealth Ave, Boston, MA 02215, United States}
\email{enfarr@bu.edu}
\author[G. Fries]{Galileo Fries}
\author[L. Garc\'{\i}a Puente]{Luis David Garc\'{\i}a Puente}
\address{Department of Mathematics and Computer Science, Colorado College, Colorado Springs, CO 80946, United States}
\email{g\_fries@coloradocollege.edu, lgarciapuente@coloradocollege.edu}
\author[J. Hutchins]{Julian Hutchins}
\address{Department of Mathematics, Morehouse College, SW, Atlanta, GA 30314, United States}
\email{Julian.hutchins@morehouse.edu}
\author[V. Hoang]{Vuong Nguyen Hoang}
\address{Department of Mathematics, Wingate University, 211 E Wilson St, Wingate, NC 28174, United States}
\email{vu.nguyenhoang054@wingate.edu}

\date{\today}

\subjclass{Primary 65H10, 92E20; Secondary 12D10}

\keywords{chemical reaction network, positive steady state, mass-action kinetics, binomial ideal}

\begin{document}


\begin{abstract}
Chemical reaction network theory is a field of applied mathematics concerned with modeling chemical systems, and can be used in other contexts such as in systems biology to study cellular signaling pathways or epidemiology to study the effect of human interaction on the spread of disease. In this paper, we seek to understand a chemical reaction network's equilibrium points through the lens of algebraic geometry by computing the positive part of the steady-state variety defined by polynomial equations arising from the assumption of mass-action kinetics. We provide a systematic classification of positive steady-state varieties produced by 2-species, 2-reaction networks, grounded in combinatorial and algebraic properties. While some (restricted) techniques exist to fully understand the ideal defining the positive steady-state variety, this computation presents a significant challenge in general. Our classification theorems provide a simplification of previous criteria, and aim to provide a foundation for future analysis of larger networks.

\end{abstract}

\maketitle

\section{Introduction}

Chemical reaction network theory is utilized in many areas of science including systems biology, epidemiology, and ecology to model systems where objects interact to form new products. For example, the authors of \cite{SGMGSG2011mbn} study reaction networks occurring in yeast which exhibit switch-like behavior, finding the smallest networks with the property of bistability. Many biological networks are bistable, meaning they have two distinct stable equilibrium points. Understanding networks' equilibrium points, or steady-states, has been a major focus of mathematical chemical reaction network theory. When viewing networks as systems which change over time, these steady-states are conditions which result in no net change in species' concentrations over time. Various theories of chemical kinetics can be used to represent how concentrations change, resulting in different mathematical models. Following the assumption of mass-action kinetics, a chemical reaction network can be translated into a system of polynomial differential equations whose solution set is an algebraic variety.

While many different areas of mathematics can be used to analyze chemical reaction networks to understand their various properties---including graph theory, combinatorics, and dynamical systems---algebraic geometry provides valuable tools to study the variety produced by networks' corresponding systems of equations. Traditional modeling approaches require numerical data that is often difficult or impossible to determine in complex biological systems; algebraic geometry provides a more flexible approach that analyzes networks' mathematical properties independent of any specific numerical values, concisely explained in \cite{Dickenstein2020tools}. For a more detailed discussion of algebraic geometry's use in chemical reaction network theory, the types of properties it can illuminate, and the types of biological networks it can model, see \cite{dickenstein2016invitation}.

After translating a network into a corresponding system of equations, the solution set of that system describes all its equilibrium points. However, the biologically relevant portions are points corresponding to steady-states with positive concentrations of all reacting species. This set is known as the positive steady-state locus, which is not necessarily an algebraic variety but rather is a semi-algebraic set. In this paper, we instead examine the components of the steady-state variety that nontrivially intersect the positive orthant, called the positive steady-state variety. The analysis of this variety requires techniques from real algebraic geometry that, in some restricted cases, may give a full description of the defining ideal. However, this computation continues to pose a significant challenge in general.

In this paper, we focus on small networks: those with two species, two reactions, and at-most-bimolecular complexes. We identify which two-reaction networks are capable of supporting positive steady-states (Theorem \ref{thm:nonempty_pssv}). When studying networks with two reactions, the corresponding equations are binomial, giving rise to binomial ideals and toric steady-state varieties. Other researchers have focused on networks with toric steady-states, notably in \cite{MDSC2012toric}; the conditions provided in Theorem \ref{thm:nonempty_pssv} are simplified significantly to better fit the 2-reaction case, and were developed independently. While most biological networks contain many more than two reactions, toric varieties do appear in biological contexts, including multi-site phosphorylation networks as studied in \cite{Adamer2020toricfamilies}.

Additionally, beyond identifying which small networks support positive steady-states, we classify 2-species small networks by the type of variety produced by giving additional conditions on the networks' reactant complexes (Sections \ref{sec:bimolec} and \ref{sec:semicubical}), and begin to build techniques to extend these classifications to networks with complexes of higher molecularity (Section \ref{sec:overlap}). The classification theorems in this paper rely on observations of the network's reactant complexes and stoichiometric matrix, which is simple to construct directly from the network. These conditions eliminate the need for potentially computationally expensive calculations of polynomial systems' solution sets or the more involved matrix calculations used in \cite{MDSC2012toric}.

\section{Background}\label{sec:CRN}

To understand the mathematical definition of chemical reaction networks, it is useful to follow an example to see how a chemical reaction network gives rise to a polynomial dynamical system. Consider a closed system containing two types of objects, or \textit{species}, labeled $A$ and $B$ whose interactions are represented diagrammatically as follows: 

$$\begin{tikzcd}
    A \ar[r,"\kappa_1", shift left] & 2B \ar[l, "\kappa_2", shift left]
\end{tikzcd}$$

This pictorial representation gives that within this closed system, one copy of $A$ produces two copies of $B$, and two copies of $B$ react to produce one copy of $A$. The \textit{complexes} are $A$ and $2B$; these are the nonnegative integer combinations of the species. The \textit{reactions} are represented as labeled arrows between complexes. The labels $\kappa_i$ are called the \textit{reaction rate constants} which can be thought of as governing the speed at which each reaction occurs. In this paper, we treat the rate constants as fixed (positive, real) unknowns in order to analyze the entire family of dynamical systems that arise from a given network as the rate constants vary.
With the goal of understanding how much of each species is present in the system over time---in other words, measuring the \textit{concentration} of each species---we examine the net change of each species in each reaction. For example, in the reaction \begin{tikzcd} A \ar[r, "\kappa_1"] & 2B \end{tikzcd}, consider species $A$. There is one copy of $A$ at the tail of the reaction arrow, in the \textit{reactant complex}, and zero copies of $A$ at the tip of the arrow, in the \textit{product complex}. So, $A$ experiences a net change in concentration of $-1$ when this reaction occurs once. Species $B$, however, experiences a net change in concentration of $+2$. These integer values, called the \textit{stoichiometric coefficients}, are often arranged into a matrix:
$$N= \begin{bNiceMatrix}[first-row, first-col] & A \rightarrow 2B & 2B \rightarrow A \\ A & -1 & \phantom{-}1 \\ B & \phantom{-}2 & -2 \end{bNiceMatrix}$$

This is called the \textit{stoichiometric matrix} of the network. As indicated above, each row represents the net change of a given species over all reactions, and each column represents the net change of all species over a given reaction. 

To understand how species' concentrations change over time within a network, we combine the stoichiometric matrix with the assumption of mass-action kinetics to generate a system of polynomial differential equations. Under the assumption of mass-action kinetics, a reaction's output rate is proportional to the product of the concentrations of the reacting species. Each species' change in concentration over time is then a function of the species' concentrations, where each reaction contributes a term. Let $x_A(t)$ and $x_B(t)$ denote the concentrations of $A$ and $B$ respectively at time $t$. The forward reaction \begin{tikzcd} [column sep=small] A \ar[r, "\kappa_1"] & 2B \end{tikzcd} contributes the term $-\kappa_1 x_A$ to the change in concentration of $A$, and the term $2\kappa_1 x_B$ to the change in concentration of $B$. These are monomial terms constructed by multiplying the stoichiometric coefficient, the reaction rate constant, and the reacting species' concentrations. Examining the entire network, mass-action kinetics give rise to the following system of polynomial ordinary differential equations: 
\begin{align*}
    \frac{dx_A}{dt}&= f_A(x) := -\kappa_1 x_A +\kappa_2 x_B^2 \\
    \frac{dx_B}{dt}&= f_B(x) := 2\kappa_1 x_A -2\kappa_2 x_B^2
\end{align*}

In general, a \textit{chemical reaction network} $G=(\mathcal{S},\mathcal{C},\mathcal{R})$ is made up of three finite sets: first, a set of species $\mathcal{S}=\{A_1,A_2,\dots, A_s\}$; second, a set of complexes $\mathcal{C}=\{y_1,y_2,\dots,y_p\}$, where each $y_i$ is a formal sum of the species with nonnegative integer coefficients; and finally a set of reactions $\mathcal{R} \subseteq (\mathcal{C} \times \mathcal{C})\setminus \{(y,y) \mid y\in \mathcal{C}\}$, which are non-diagonal ordered pairs of complexes such that $(y_i, y_j)$ corresponds to the reaction graphically represented as $y_i\rightarrow y_j$. Each reaction also has a positive real rate constant $\kappa_i$. As in the above example, we construct the \textit{steady-state polynomials}, one for each species, where each reaction contributes one monomial term.

To build the steady-state polynomials in general, we require some additional definitions and notation. Writing out the $i^{\text{th}}$ complex gives $y_i := \alpha_{1i}A_1+\alpha_{2i}A_2+\dotsm +\alpha_{si}A_s$, where the $\alpha_{ki}$ are the coefficients on the $k^{\text{th}}$ species in the $i^{\text{th}}$ complex. We collect information from these coefficients in the stoichiometric matrix $N$, whose rows correspond to species and whose columns correspond to reactions. The entry in the row corresponding to the $k^{\text{th}}$ species and the column corresponding to the reaction $(y_i, y_j)$ is the difference $\alpha_{kj}-\alpha_{ki}$, which is the net change in species $A_k$. Next, we construct a diagonal matrix of the reaction rates $\text{diag}(\kappa)$ whose nonzero entries are the reaction rates $\kappa_i$. Finally, we construct a vector $x^B$ of monomials in the species' concentrations. Here $B$ is the \textit{reactant matrix}, where again each column corresponds to a reaction and each row corresponds to a species; the column corresponding to the reaction $(y_i, y_j)$ lists the $\alpha_{ki}$ coefficients of the reactant complex $y_i$ for $k=1,\dots,s$. If $b_i = (\alpha_{1i}, \ldots, \alpha_{si})$ is the $i^\text{th}$ column of $B$, then $x^{b_i} = x_{A_1}^{\alpha_{1i}} \cdots x_{A_s}^{\alpha_{si}}$ is the $i^\text{th}$ entry of $x^B$. Then, 
the dynamical system arising from the chemical reaction network $G$ and a choice of rate parameters $\kappa$ is given by
\begin{equation}\label{mass-action-system}
\frac{dx}{dt} =  N \cdot \text{diag}(\kappa) \cdot x^B =: f_\kappa(x) = (f_\kappa(x)_1, f_\kappa(x)_2,\dots,f_\kappa(x)_s)
\end{equation}
The steady-state polynomials are the expressions on the right hand side of Equation \ref{mass-action-system}, and the \emph{steady-state equations} are obtained by setting the steady-state polynomials equal to zero. The solutions to the system of steady-state equations are the  equilibrium points of the network. These values may or may not depend on the specific values of the rate constants, discussed below in Section \ref{sec:indep}.

\begin{definition} A \textbf{steady-state} of a chemical reaction network (together with fixed rate constants) is a tuple of species concentrations $x = (x_{A_1}, x_{A_2}, \dots ,x_{A_s})\in\mathbb{R}^s$ that satisfies all
the steady-state equations. 
\end{definition}

Intuitively, steady-states are inputs that will result in no change in the species' concentrations over time. The solution set of the system of steady-state equations forms a complex algebraic variety called the \textit{steady-state variety}. In the opening example network of Section \ref{sec:CRN}, the two steady-state polynomials are constant multiples of each other, and so their individual zero sets are identical. Setting either equation equal to zero gives that the steady-state variety is the curve $x_A=\frac{\kappa_2}{\kappa_1}x_B^2$, which is a parabola opening in the positive $x_A$ direction. 

Of particular interest is the \textit{positive steady-state variety}, which is the Zariski closure of the intersection of the steady-state variety with the interior of the positive orthant of $\mathbb{R}^s$. This variety represents all biologically relevant steady-states. In our example, the positive steady-state variety is the same as the steady-state variety. However, in general, there is no algorithm to find the positive real part of the steady-state variety of a network, so our approach is to restrict our focus to a specific category of small chemical reaction networks.

Networks are sometimes grouped together based on their size. One notion of size is the number of species and the number of reactions in the network. In this paper, we consider networks with 2 species and 2 reactions. Additionally, we require that both species appear at least once among the network's complexes. Networks with this extra property are defined as \textit{genuine} networks, as in \cite{NetworkEnumeration}. To capture a different notion of size, we consider the network's complexes:

\begin{definition}
    The \textbf{molecularity} of a complex $y_i = \alpha_{i1}A_1+\alpha_{i2}A_2+\dotsm +\alpha_{is}A_s$ is the sum $\sum_{j=1}^s \alpha_{ij}$.  
\end{definition}

Intuitively, this measures the number of total objects appearing in a complex. A network can then be characterized by whether its complexes satisfy a fixed upper bound on their molecularity. If this upper molecularity bound is $n$, an \textit{at-most-n-molecular} network gives rise to steady-state polynomials whose degree is bounded above by $n$. For example, the above network of \begin{tikzcd} A \ar[r,"\kappa_1", shift left] & 2B \ar[l, "\kappa_2", shift left] \end{tikzcd} is \textit{at-most-bimolecular}, since the largest complex $2B$ contains two objects (and thus the degree of its steady-state polynomials is at most two). Most complexes considered in this paper are either unimolecular, bimolecular, or trimolecular, meaning they contain one, two, or three objects. We focus mainly on the at-most-bimolecular case. As such, a genuine, at-most-bimolecular, 2-species, 2-reaction chemical reaction network will henceforth be referred to as a ``small network," with different molecularity bounds or numbers of species specifically indicated as necessary. For a given $s$-species small network, the 
general forms of its reactions, steady-state polynomials, and stoichiometric matrix are given by the following formulas, respectively:

\begin{equation} \label{general22}
\begin{tikzcd}[row sep=tiny] \alpha_{10} A_1 + \alpha_{20} A_2 + \dotsm + \alpha_{s0} A_s \ar[r, "\kappa_1"] & \alpha_{11} A_1 + \alpha_{21} A_2 + \dotsm + \alpha_{s1} A_s \\ \alpha_{12} A_1 + \alpha_{22} A_2 + \dotsm + \alpha_{s2} A_s \ar[r, "\kappa_2"] & \alpha_{13} A_1 + \alpha_{23} A_2 + \dotsm + \alpha_{s3} A_s \end{tikzcd}
\end{equation}

\begin{equation}\label{generalSS}
f(x)=\begin{cases}
f_{A_1}(x) &= (\alpha_{11}-\alpha_{10})\kappa_1 x_{A_1}^{\alpha_{10}} x_{A_2}^{\alpha_{20}}\cdots x_{A_s}^{\alpha_{s0}} + (\alpha_{13}-\alpha_{12})\kappa_2 x_{A_1}^{\alpha_{12}} x_{A_2}^{\alpha_{22}}\cdots x_{A_s}^{\alpha_{s2}}\\
f_{A_2}(x) &= (\alpha_{21}-\alpha_{20})\kappa_1 x_{A_1}^{\alpha_{10}} x_{A_2}^{\alpha_{20}}\cdots x_{A_s}^{\alpha_{s0}} + (\alpha_{s3}-\alpha_{s2})\kappa_2 x_{A_1}^{\alpha_{12}} x_{A_2}^{\alpha_{22}}\cdots x_{A_s}^{\alpha_{s2}}\\
&\vdots\\
f_{A_s}(x) &= (\alpha_{s1}-\alpha_{s0})\kappa_1 x_{A_1}^{\alpha_{10}} x_{A_2}^{\alpha_{20}}\cdots x_{A_s}^{\alpha_{s0}} + (\alpha_{s3}-\alpha_{s2})\kappa_2 x_{A_1}^{\alpha_{12}} x_{A_2}^{\alpha_{22}}\cdots x_{A_s}^{\alpha_{s2}}
\end{cases}
\end{equation}

\begin{equation}\label{generalSM}
N=\begin{bmatrix}
    \alpha_{11}-\alpha_{10} & \alpha_{13}-\alpha_{12} \\ \alpha_{21}-\alpha_{20} & \alpha_{23}-\alpha_{22} \\ \vdots & \vdots \\ \alpha_{s1}-\alpha_{s0} & \alpha_{s3}-\alpha_{s2}
\end{bmatrix}
\end{equation}

In the 2-species case, the species will be labeled $A$ and $B$ with corresponding stoichiometric coefficients $\alpha_i$ and $\beta_i$ and corresponding variables $x$ and  $y$ for simplicity of notation.

\section{Conditions for a Nonempty Positive Steady-State Variety}

When examining 2-species small networks, all networks that support a nonempty positive steady-state variety share a common property in their stoichiometric matrix: the columns are negative multiples of each other. Indeed, this property is shared by all 2-reaction networks with nonempty positive steady-state varieties. In the following proof, $x_{A_i}=:x_i$ for simplicity of notation, and the terms ``positive" or ``negative" denote values that are strictly greater or less than zero, respectively.

\begin{theorem} \label{thm:nonempty_pssv}
    For generic choices of rate constants, a two-reaction network with any number of species and with complexes of any molecularity produces a nonempty positive steady-state variety if and only if the columns of its stoichiometric matrix are (nonzero) negative constant multiples of each other, and its reactant complexes are not identical.
\end{theorem}

\begin{proof} 
$(\Rightarrow)$ We give a proof by contrapositive, showing that if either condition does not hold then the positive steady-state variety is empty. First, suppose that the two reactant complexes are identical and are $\alpha_{1} A_1+\alpha_2 A_2 + \dotsm + \alpha_s A_s$. Then, after factoring, all $s$ steady-state equations take the form $cx_1^{\alpha_1}x_2^{\alpha_2}\cdots x_s^{\alpha_s}=0$ for some (possibly different) real constants $c$. It is possible for $c$ to equal zero, discussed in Section \ref{sec:indep} below. However, for generic choices of rate constants this equation only has solutions when at least one variable equals zero, corresponding to coordinate hyperplane solutions. Recall that the steady-state variety is the intersection of the individual zero sets of the steady-state equations, meaning the variety is therefore also restricted to the coordinate hyperplanes, and the  positive steady-state variety is empty. So, if a network has a nonempty positive steady-state variety, its reactant complexes must be non-identical. 

Now, suppose that the columns of the stoichiometric matrix are not negative multiples of each other. There are two cases to consider, as the columns being multiples of each other by zero implies that one of the reactions is trivial. First, consider the case where the columns are positive multiples of each other. This means the rank of the stoichiometric matrix is one, and so the rows must all be multiples of each other as well. The variable and rate constant appearing in each monomial term are fixed by the reaction, so the $s$ steady-state equations differ only by a constant real factor and we need only consider one equation when examining the steady-state variety. Without loss of generality, suppose that $f_{A_1}$ is nonzero and consider its general form
\begin{align*} f_{A_1} &= (\alpha_{11}-\alpha_{10})\kappa_1 x_{1}^{\alpha_{10}} x_{2}^{\alpha_{20}}\cdots x_{s}^{\alpha_{s0}} + (\alpha_{13}-\alpha_{12})\kappa_2 x_{1}^{\alpha_{12}} x_{2}^{\alpha_{22}}\cdots x_{s}^{\alpha_{s2}}\\
&=c_0x_{1}^{\alpha_{10}} x_{2}^{\alpha_{20}}\cdots x_{s}^{\alpha_{s0}}+c_2x_{1}^{\alpha_{12}} x_{2}^{\alpha_{22}}\cdots x_{s}^{\alpha_{s2}} \end{align*}
where the (nonzero) real constants $c_0$ and $c_2$ encode the reaction rates and the stoichiometric coefficients. The steady-state variety is then given by
$$-x_{1}^{\alpha_{10}} x_{2}^{\alpha_{20}}\cdots x_{s}^{\alpha_{s0}}=\frac{c_2}{c_0}x_{1}^{\alpha_{12}} x_{2}^{\alpha_{22}}\cdots x_{s}^{\alpha_{s2}}.$$
The reaction rates are always positive and the stoichiometric coefficients have the same sign by assumption, so $c_0$ and $c_2$ must have the same sign; thus the fraction $\frac{c_2}{c_0}$ is positive. We wish to show that the $x_i$ cannot all be simultaneously positive. Suppose without loss of generality that $x_1,\dots, x_{s-1}$ are positive. Then, the sign of both sides is determined entirely by the value of the powers of $x_s$. For equality to hold, the two sides must first have the same sign, which can clearly only happen if $x_s$ is negative or zero. Thus, no positive steady-states can exist. 

Now, consider the case where the columns are not multiples of each other. Then, the rank of the matrix is two, giving that there exists at least one pair of nonzero rows that are linearly independent and therefore not multiples of each other. If any steady-state polynomial is a monomial, the positive steady-state variety is empty. However, both steady-state polynomials corresponding to a pair of linearly independent rows could be binomials. Suppose without loss of generality that the first two rows are not multiples of each other, and consider the general form steady-state polynomials for $A_1$ and $A_2$:
\begin{align*} f_{A_1} &= (\alpha_{11}-\alpha_{10})\kappa_1 x_{1}^{\alpha_{10}} x_{2}^{\alpha_{20}}\cdots x_{s}^{\alpha_{s0}} + (\alpha_{13}-\alpha_{12})\kappa_2 x_{1}^{\alpha_{12}} x_{2}^{\alpha_{22}}\cdots x_{s}^{\alpha_{s2}}\\
f_{A_2} &= (\alpha_{21}-\alpha_{20})\kappa_1 x_{1}^{\alpha_{10}} x_{2}^{\alpha_{20}}\cdots x_{s}^{\alpha_{s0}} + (\alpha_{23}-\alpha_{22})\kappa_2 x_{1}^{\alpha_{12}} x_{2}^{\alpha_{22}}\cdots x_{s}^{\alpha_{s2}} \end{align*}

where none of the stoichiometric coefficients are zero. Set equal to zero, we can rewrite $f_A$ and $f_B$ respectively as 
\begin{align*} -x_{1}^{\alpha_{10}} x_{2}^{\alpha_{20}}\cdots x_{s}^{\alpha_{s0}} &= \frac{\kappa_2(\alpha_{13}-\alpha_{12})}{\kappa_1(\alpha_{11}-\alpha_{10})}x_{1}^{\alpha_{12}} x_{2}^{\alpha_{22}}\cdots x_{s}^{\alpha_{s2}}\\
\text{  and } -x_{1}^{\alpha_{10}} x_{2}^{\alpha_{20}}\cdots x_{s}^{\alpha_{s0}} &= \frac{\kappa_2(\alpha_{23}-\alpha_{22})}{\kappa_1(\alpha_{21}-\alpha_{20})}x_{1}^{\alpha_{12}} x_{2}^{\alpha_{22}}\cdots x_{s}^{\alpha_{s2}}.\end{align*}

The fractions are real number scaling factors, affecting a stretch or shrink on the curve. The values of $\kappa_1$ and $\kappa_2$ are fixed, so the relative steepness of the two zero sets is determined by the ratios of the stoichiometric coefficients. If 
$$\frac{\alpha_{13}-\alpha_{12}}{\alpha_{11}-\alpha_{10}}=\frac{\alpha_{23}-\alpha_{22}}{\alpha_{21}-\alpha_{20}},$$
then the zero sets overlap completely. Consider the $2\times 2$ submatrix of the stoichiometric matrix corresponding to these two species:

$$\begin{bmatrix}
    \alpha_{11}-\alpha_{10} & \alpha_{13}-\alpha_{12} \\ \alpha_{21}-\alpha_{20} & \alpha_{23}-\alpha_{22} 
\end{bmatrix}$$

By assumption, these two rows of the stoichiometric matrix are linearly independent, so this submatrix has rank two. But, the above equality implies that the columns are multiples of each other, giving a rank of one, meaning that the zero sets of $f_{A_1}$ and $f_{A_2}$  must have different scaling factors. A solution to the system with all $x_i\neq 0$ would require a nonzero left-hand side of the equation to simultaneously equal two distinct nonzero values, which is impossible. So, these two zero sets can only intersect at the origin and/or in coordinate hyperplanes. The steady-state variety of the network is the intersection of the zero sets of all $s$ steady-state equations, and so is also restricted to the coordinate hyperplanes, resulting in an empty positive steady-state variety. Thus, if a network has a nonempty positive steady-state variety, the columns of its stoichiometric matrix are negative multiples of each other. 

$(\Leftarrow)$ Now, suppose that a network satisfies the above conditions. The relationship between the columns means the stoichiometric matrix has rank one, meaning the rows are all multiples of each other. Therefore, all of the steady-state equations are constant multiples of each other, since the other components of each monomial term are fixed by the reaction. All of the nonzero equations thus define the same zero set, and we need only consider one equation to fully understand the steady-state variety. Without loss of generality, suppose that $f_{A_1}$ is nonzero. In general form,

\begin{align*} f_{A_1} &= (\alpha_{11}-\alpha_{10})\kappa_1 x_{1}^{\alpha_{10}} x_{2}^{\alpha_{20}}\cdots x_{s}^{\alpha_{s0}} + (\alpha_{13}-\alpha_{12})\kappa_2 x_{1}^{\alpha_{12}} x_{2}^{\alpha_{22}}\cdots x_{s}^{\alpha_{s2}}\\
&=c_0x_{1}^{\alpha_{10}} x_{2}^{\alpha_{20}}\cdots x_{s}^{\alpha_{s0}}+c_2x_{1}^{\alpha_{12}} x_{2}^{\alpha_{22}}\cdots x_{s}^{\alpha_{s2}}, \end{align*}

where the (nonzero) $c_i$ encode the rate constants and stoichiometric coefficients. We wish to verify the existence of at least one solution where all of the $x_i$ are simultaneously positive. Setting equal to zero, the steady-state variety is defined by

$$x_{1}^{\alpha_{10}} x_{2}^{\alpha_{20}}\cdots x_{s}^{\alpha_{s0}}=-\frac{c_2}{c_0}x_{1}^{\alpha_{12}} x_{2}^{\alpha_{22}}\cdots x_{s}^{\alpha_{s2}}.$$

The value of $-\frac{c_2}{c_0}$ is some positive real constant because the $c_i$ have opposite signs by assumption. To understand the possible values of the exponents, consider the stoichiometric matrix

$$N=\begin{bmatrix}
    \alpha_{11}-\alpha_{10} & \alpha_{13}-\alpha_{12} \\ \alpha_{21}-\alpha_{20} & \alpha_{23}-\alpha_{22} \\ \vdots & \vdots \\ \alpha_{s1}-\alpha_{s0} & \alpha_{s3}-\alpha_{s2}
\end{bmatrix}$$

The $\alpha_{ij}$ are all nonnegative integers. The values where $j$ is even correspond to our exponents and cannot all be zero, as this would result in an entirely nonnegative stoichiometric matrix, contradicting our assumption. Additionally, if $\alpha_{10}=0$, then $\alpha_{12}\neq0$ because $\alpha_{13}-\alpha_{12}$ must be negative by assumption. Similarly, if $\alpha_{12}=0$ then $\alpha_{10}\neq0$. Without loss of generality, suppose that $\alpha_{10}\neq 0$. Returning to the steady-state equation, we can group the variables:
$$x_1^{\alpha_{10}-\alpha_{12}}=-\frac{c_2}{c_0}x_2^{\alpha_{22}-\alpha_{20}}\cdots x_s^{\alpha_{s2}-\alpha_{s0}}$$
This division has the potential to eliminate shared copies of a given variable from both monomial terms of the steady-state equation, which corresponds to removing coordinate hyperplane solutions which are never positive, and therefore is appropriate. By our assumption that the reactant complexes are not identical, these new exponents cannot all be simultaneously zero, so at least one of the $x_i$ will always have a nonzero exponent. Suppose without loss of generality that $\alpha_{10}-\alpha_{12}\neq 0$, and consider some point in $\mathbb{R}^s$ where $x_i=b_i$ for some positive constant $b_i$ for all $i\neq 1$. Then, the equation becomes
$$x_1^{\alpha_{10}-\alpha_{12}}=-\frac{c_2}{c_0}b_2^{\alpha_{22}-\alpha_{20}}\cdots b_s^{\alpha_{s0}-\alpha_{s2}}\Rightarrow x_1=\left(-\frac{c_2}{c_0}b_2^{\alpha_{22}-\alpha_{20}}\cdots b_s^{\alpha_{s0}-\alpha_{s2}}\right)^{\frac{1}{\alpha_{10}-\alpha_{12}}}$$
The right hand side of the equation is a positive real number, as it consists only of positive numbers raised to different powers. So, a solution exists to the equation defining the steady-state variety such that all variables are simultaneously positive, meaning that the steady-state variety intersects the interior of the positive orthant and thus the positive steady-state variety is nonempty. 
\end{proof}

\begin{remark}
Theorem \ref{thm:nonempty_pssv} is related (but different) to Lemma 2.7 in \cite{16M1069705}. This lemma implies that given a $2$-reaction network, the resulting mass-action system has at least one positive steady-state in some \emph{stoichiometric compatibility class} if and only if its reaction vectors are negative scalar multiples of each other. 
\end{remark} 

\begin{remark}
The condition on the columns of the stoichiometric matrix can also be described with the notion of T-alternating subnetworks described in \cite{Obatake_2020}. Through this lens we instead require that, when viewing the behavior of each species individually, the network is 1-alternating.
\end{remark} 

\begin{remark}
The steady-state ideals of all 2-reaction networks are binomial ideals, since the corresponding steady-state polynomials are binomials; so, the steady-state varieties (when irreducible) are \emph{toric varieties}. Networks with toric nonempty positive steady-state varieties have been studied previously, most notably in \cite{MDSC2012toric}. The authors give three conditions to classify networks with toric nonempty positive steady-state varieties. Their Condition $3.1$ is unnecessary in the 2-reaction case, as the varieties are necessarily toric. Furthermore, 2-reaction networks trivially satisfy Condition $3.6$, as the matrix $\Delta$ from \cite{MDSC2012toric} has full column rank; this condition also encapsulates the second condition above by removing solutions contained in the coordinate hyperplanes. Finally, their Condition $3.4$ is equivalent to having the columns of the stoichiometric matrix be (nonzero) negative multiples of each other, the first condition above. Then, 2-reaction networks can be considered a special case of the toric networks discussed in \cite{MDSC2012toric}, and the above theorem an independently derived simplification of their conditions for small networks to have nonempty positive steady-state varieties.
\end{remark}

\subsection{Reaction Rate Independence}\label{sec:indep}

In some reaction networks, the specific values of the rate constants do not significantly impact the shape of the positive steady-state variety. They may alter the slope of a line or the steepness of a parabola, but the overall form and existence remains the same. For example, the positive steady-state variety of the network presented in Section \ref{sec:CRN} is a parabola for all choices of rate constants. However, this is not always the case, as seen in the following 1-species, 3-reaction network:

\begin{example} Consider the one-species network in Figure \ref{fig:one_species_network}. The steady-state variety is defined by the polynomial $f_A=\kappa_1x_A-\kappa_2x_A-\kappa_3x_A^2$.
Factoring and setting equal to zero to find the steady-state variety, we see that $0=x_A(\kappa_1-\kappa_2-\kappa_3x_A)$. The steady-state variety consists of the point $x_A=0$ and the point $x_A=\frac{\kappa_1-\kappa_2}{\kappa_3}$. The three rate constants are always positive, meaning that the second point in the variety will be positive only when $\kappa_1>\kappa_2$.
\end{example}

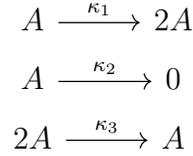
\begin{figure}[ht]
    \centering
    \begin{tikzcd} [row sep = tiny]
    A \ar[r,"\kappa_1"] & 2A \\
    A \ar[r,"\kappa_2"] & 0 \\
    2A \ar[r,"\kappa_3"] & A
    \end{tikzcd}
    \caption{1-species, 3-reaction network exhibiting rate dependence}
    \label{fig:one_species_network}
\end{figure}

In other words, the existence of a nonempty positive steady-state variety depends on the specific values of the rate constants. This problem of rate-dependence greatly complicates the study of positive steady-state varieties. Rate-dependence in the 2-reaction case is discussed in the following corollary.

\begin{corollary}\label{cor:2r_rate_indep}
    If a small network with complexes of any molecularity has non-identical reactant complexes, the existence of a nonempty positive steady-state variety is independent of the specific choices of reaction rates. For a small network with identical reactant complexes, there exist rate constants giving a nonempty positive steady-state variety if and only if it satisfies the other condition of Theorem \ref{thm:nonempty_pssv}; for that choice of rate constants, the positive steady-state variety is $\mathbb{R}_{>0}$.
\end{corollary}

\begin{proof}
Consider some small network with non-identical reactant complexes that has a nonempty positive steady-state variety for some specific values of the rate constants. As in the forwards direction of Theorem \ref{thm:nonempty_pssv}, this network has a stoichiometric matrix whose columns are negative multiples of each other. Following the logic from the backwards direction of Theorem \ref{thm:nonempty_pssv}, the steady-state variety is defined by
$$x_{1}^{\alpha_{10}}x_{2}^{\alpha_{20}}\cdots x_{s}^{\alpha_{s0}} = -\frac{\kappa_2(\alpha_{23}-\alpha_{22})}{\kappa_1(\alpha_{21}-\alpha_{20})}x_{1}^{\alpha_{12}} x_{2}^{\alpha_{22}}\cdots x_{s}^{\alpha_{s2}}.$$
Since the $\kappa_i$ are always positive and the stoichiometric coefficients have opposite signs, the coefficient on the right hand side is positive regardless of the exact values of the rate constants. Call the value of this fraction $c$. We wish to show that a positive solution exists. To this end, we can divide the like variable terms: 
$$x_1^{\alpha_0-\alpha_2}=cx_2^{\alpha_{22}-\alpha_{20}}\cdots x_s^{\alpha_{s2}-\alpha_{s0}}$$
The logic is identical to that presented at the end of Theorem \ref{thm:nonempty_pssv}; a positive steady-state exists, and its existence did not depend on the particular values of the rate constants. 

If the two reactant complexes are identical, then each steady-state polynomial takes the form 
$$f_{A_i}=c_ix_1^{\alpha_1}x_2^{\alpha_2}\dots x_s^{\alpha_s}=((\alpha_{i1}-\alpha_{i0})\kappa_1+(\alpha_{i3}-\alpha_{i2})\kappa_2)x_1^{\alpha_1}x_2^{\alpha_2}\dots x_s^{\alpha_s}$$
First, suppose the network has a nonempty positive steady-state variety. Then, there must exist a solution where all variables are simultaneously positive. This occurs only when $c_i=0$, in which case we have
$$\kappa_1=-\frac{(\alpha_{i3}-\alpha_{i2})}{(\alpha_{i1}-\alpha_{i0})}\kappa_2$$
Since both rate constants are positive, the stoichiometric coefficients must have opposite signs. Additionally, we must have $c_i=0$ for all $i$, so the ratio between the stoichiometric coefficients must be constant across all species. These two conditions describe a network whose stoichiometric matrix has columns which are negative (nonzero) constant multiples of each other, which is the other condition of Theorem \ref{thm:nonempty_pssv}. Now, suppose the columns of the stoichiometric matrix are negative multiples of each other by some constant $-c$. Then, rate constants following $\kappa_1=c\kappa_2$ will make $c_i=0$ for all $i$. In both cases, proper choices of the $\kappa_i$ make $f_{A_i}=0$ for all $i$ and the positive steady-state variety is $\mathbb{R}_{>0}$.
\end{proof}

\section{2-Species Small Networks} \label{sec:bimolec}

Focusing now on small networks which are indeed at-most-bimolecular, there are 210 non-isomorphic networks, extracted from data available at \cite{NetworkEnumeration} using the code available at \url{https://github.com/elnfarr/Algebraic-Goemetry-of-Small-CRN}. Before classifying the positive steady-state varieties produced by these networks, we require one additional piece of vocabulary:

\begin{definition}
    The \textbf{support} of a complex is the set of all species that appear within it.
\end{definition}
\begin{example}
    The support of the complex $2A+B$ is the set $\{A,B\}$. The support of the complex $3B$ is $\{B\}$.
\end{example}

Among the small networks, four classes of nonempty positive steady-state varieties arose following a systematic calculation of every steady-state variety: non-axis horizontal or vertical lines, slanted lines through the origin, parabolas, and hyperbolas. 

\begin{theorem} Of the 210 non-isomorphic small networks, three networks have a horizontal or vertical line as their positive steady-state variety; seven networks have slanted lines; five networks have parabolas; and and three networks have hyperbolas. 
\end{theorem}

This theorem was proved computationally; the list of network calculations is available at the GitHub page linked above. In addition to satisfying the conditions of Theorem \ref{thm:nonempty_pssv}, the networks producing each class of positive steady-state variety share common properties in their reactant complexes. The reactant complex pairs satisfying these properties are non-identical, so by Corollary \ref{cor:2r_rate_indep} the form of the positive steady-state variety is independent of the specific values of the rate constants. The following theorem classifies small networks according to the class of positive steady-state variety they produce.

\begin{theorem} \label{thm:bimolec}
Given a small network satisfying the assumptions of Theorem \ref{thm:nonempty_pssv}, the positive steady-state variety is:
\begin{itemize} 
\item a non-axis horizontal or vertical line if and only if one reactant complex is $A+B$ and the other is unimolecular; or,
\item a line through the origin if and only if 1) The two reactant complexes have the same molecularity, and 2) The supports of the reactant complexes are nonempty and distinct (not necessarily disjoint); or,
\item a parabola if and only if 1) One reactant complex is bimolecular and the other is unimolecular, and 2) The supports of the reactant complexes are disjoint; or,
\item a hyperbola if and only if the reactant complexes are $A+B$ and $0$.
\end{itemize}
\end{theorem}

\begin{proof}
($\Rightarrow$) First, consider the two linear classes of positive steady-state variety. Suppose we have an arbitrary small network whose positive steady-state variety is a non-axis horizontal or vertical line, or a line through the origin. Up to relabeling species, the simplified positive steady-state varieties are $V(x-c)$ and $V(y-cx)$ respectively for some positive $c\in\mathbb{R}$, where $V(f(x))$ denotes the variety generated by $f(x)=0$. Since the positive steady-state variety is a component of the steady-state variety, the equations defining the steady-state variety must have $(x-c)$ or $(y-cx)$ as a factor, but may also contain another factor term whose degree is limited by the network's bimolecularity:
$$(x-c)\cdot(\text{Factor of degree 1 or 0}), \hspace{1cm} (y-cx)\cdot(\text{Factor of degree 1 or 0})$$
Since the network has two reactions the corresponding equations have at most two terms, limiting the degree-one factors to be monomials or $(x+c)$ and $(y+cx)$ respectively, and there are four potential equations defining the steady-state variety in each case. We also have that the steady-state variety equals $V(f_A)\cap V(f_B)$, so the same logic applies to the individual steady-state equations. As in Theorem \ref{thm:nonempty_pssv} the two steady-state equations are multiples of each other and take on the same form. So, up to multiplication by constants, the steady-state polynomials in each case take on one of four possible forms. 

Consider first the horizontal or vertical line class of networks. The possible forms of the steady-state polynomials are as follows:
$$f_i=x-c\text{ ; }f_i=x^2-c^2\text{ ; }f_i=x^2-cx\text{ ; }f_i=xy-cy$$
By Theorem \ref{thm:nonempty_pssv}, the columns of the network's stoichiometric matrix are negative multiples of each other. This restriction necessitates that the first three forms of equations correspond to non-genuine networks: $A \longleftrightarrow 0$, $2A\longleftrightarrow 0$, and $2A \longleftrightarrow A$ respectively, as any copies of $B$ in the network must be in the product complexes which can only result in positive entries of the stoichiometric matrix. Thus, the steady-state polynomials only take the following form, up to multiplication by constants:
$$f_i=xy-cy$$ 
The monomial terms must therefore have variable parts of $xy$ and $y$. By the definition of the steady-state polynomials, these correspond to reactant complexes of $A+B$ and $B$. With arbitrary labeling of species, the reactant complexes are therefore $A+B$ and a unimolecular complex. 

Now, consider the slanted line class of networks. The possible forms of the steady-state polynomials are as follows:
$$f_i=y-cx\text{ ; }f_i=y^2-c^2x^2\text{ ; }f_i=y^2-cxy\text{ ; }f_i=xy-cx^2$$
These correspond to reactant complex pairs of $A$ and $B$; $2A$ and $2B$; $A+B$ and $2B$; or $A+B$ and $2A$ respectively. In all pairs the molecularities of the reactant complexes are the same, and with arbitrary labeling of species the complexes have distinct (not necessarily disjoint) supports, proving both conditions.

Next, consider the two quadratic variety classes. Suppose we have an arbitrary small network whose positive steady-state variety is a parabola or a hyperbola. Up to relabeling species, the simplified positive steady-state varieties are $V(y-cx^2)$ and $V(xy-c)$ for some positive $c\in\mathbb{R}$. The network is at-most-bimolecular, which implies that each steady-state polynomial has degree at most two. Following similar logic to the linear classes, the steady-state polynomials take on the following forms up to multiplication by constants:
$$f_i=y-cx^2\text{ for parabolas;}\hspace{0.5cm} f_i=xy-c\text{ for hyperbolas.}$$
As with the linear variety classes, the variable components of the monomial terms give that for a parabola the reactant complexes are $B$ and $2A$; with arbitrary labeling of species, these are a bimolecular and a unimolecular complex with disjoint supports, proving both conditions. Similarly, for a hyperbola the corresponding reactant complexes are $A+B$ and $0$, as desired.

($\Leftarrow$) 
For all four variety classes, the assumptions of Theorem \ref{thm:nonempty_pssv} guarantee that we need only examine one steady-state equation to completely understand the steady-state variety, since the two equations are constant multiples of each other. We consider without loss of generality a nonzero $f_A$, and proceed through each class to show that if a network meets the above conditions on the reactant complexes, the positive steady-state variety has the proposed form.

First, suppose we have a small network where one reactant complex is $A+B$ while the other is, without loss of generality, $A$. Constructing $f_A$, setting equal to zero, and solving gives
$$f_{A} = (\alpha_1-1)\kappa_1 {x}{y} + (1-\alpha_1)\kappa_2 x \hspace{0.3cm} \Longrightarrow \hspace{0.3cm} x = 0\text{, }y = -\frac{(\alpha_3-1)\kappa_2}{(1-\alpha_1)\kappa_1}$$
By assumption the stoichiometric coefficients have opposite signs, so the fixed value of $y$ is positive and therefore defines a non-axis horizontal line crossing through the first quadrant. With arbitrary labeling of species, the positive steady-state variety is a horizontal or vertical line.

Now suppose we have a small network where both reactant complexes have the same molecularity and their supports are distinct. Up to relabeling species, there are three cases to consider:
\begin{enumerate}
    \item The reactant complexes are $A$ and $B$;
    \item The reactant complexes are $2A$ and $2B$;
    \item The reactant complexes are $2A$ and $A+B$.
\end{enumerate}
These give rise to the following:
\begin{align*}
     f_{A1}=\kappa_1(\alpha_1-1)x+\kappa_2(\alpha_3)y \hspace{1cm} &\Longrightarrow \hspace{1cm} x=-\frac{\kappa_2\alpha_3}{\kappa_1(\alpha_1-1)}y\\    
     f_{A2}=\kappa_1(\alpha_1-2)x^2+\kappa_2(\alpha_3)y^2 \hspace{1cm} &\Longrightarrow \hspace{1cm} x=\pm\sqrt{\frac{-\kappa_2\alpha_3}{\kappa_1(\alpha_1-2)}}y\\    
     f_{A3}=\kappa_1(\alpha_1-2)x^2+\kappa_2(\alpha_3-1)xy \hspace{1cm} &\Longrightarrow \hspace{1cm} x=0, \hspace{0.3cm} x=-\frac{\kappa_2(\alpha_3-1)}{\kappa_1(\alpha_1-2)}y
\end{align*}

As above, the fractional coefficients are positive and in all cases the equation defines a line with positive slope through the origin. The second and third cases define additional lines, but these do not intersect the interior of the positive orthant. Thus, the positive steady-state variety is a line through the origin.

Next suppose that the reactant complexes have disjoint supports with one bimolecular and one unimolecular complex, and without loss of generality suppose the reactant complexes are $A$ and $2B$. Then,
$$f_A=\kappa_1(\alpha_1-1)x+\kappa_2(\alpha_3)y^2\hspace{0.3cm} \Longrightarrow \hspace{0.3cm}x=-\frac{\kappa_2(\alpha_3)}{\kappa_1(\alpha_1-1)}y^2$$
Again, the coefficient on $y^2$ is positive by assumption and the equation defines a parabola crossing into the first quadrant; the positive steady-state variety is a parabola.

Finally, suppose a small network has reactant complexes of $A+B$ and $0$. Then,
$$f_B = \kappa_1(\beta_1-1)xy + \kappa_2(\beta_3)\hspace{0.3cm} \Longrightarrow \hspace{0.3cm}xy=-\frac{\kappa_2(\beta_3)}{\kappa_1(\beta_1-1)}$$
As in the other three cases, the constant value of $xy$ is positive and this equation defines a hyperbola in the first and third quadrants; the positive steady-state variety is a hyperbola.
\end{proof}

\begin{remark} Some networks described by the above theorem exhibit absolute concentration robustness (ACR), a network property first introduced by Shinar and Feinberg in 2010 \cite{ShinarFeinberg2010}. A biochemical system shows ACR if, for a given species, its concentration is the same in all positive steady-states of the system. For networks whose positive steady-state variety is a horizontal or vertical line, one species has a constant concentration throughout the variety and thus shows ACR in that species. No other variety class described above has this property. Theorem \ref{thm:bimolec} therefore completely describes all small networks which exhibit ACR. \end{remark}

\section{Degree 3 Varieties} \label{sec:semicubical}

When expanding our focus to include at-most-trimolecular 2-species small networks, two new classes of positive steady-state variety arise. Unlike the at-most-bimolecular case, the following theorems do not arise from a systematic analysis of all possible networks so the exact number of networks with each class of variety is unknown. The shape of the first class of variety is formally known as a semicubical parabola and consists of equations of the form $y^2=cx^3$. Informally, it looks like a cartoon bird flying through the 2D plane (Figure \ref{fig:semicubic_ex}).

\begin{example} Consider the network and corresponding steady-state polynomials and variety in Figure \ref{fig:semicubic_ex}. Here, $\kappa_{1} = \kappa_{2} = 1$ with $x_{A}$ as the horizontal axis and $x_{B}$ as the vertical axis. The steady-state variety is a semicubial parabola (shown in red).

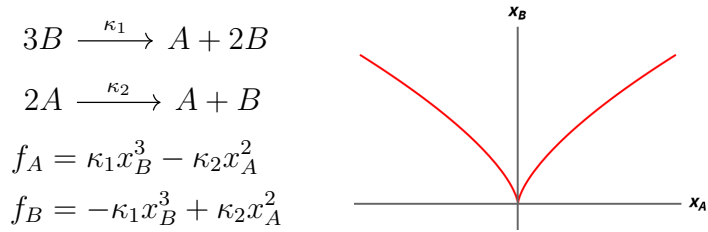
\begin{figure}[ht]
    \begin{subfigure}{0.3\textwidth}
    \centering
    \begin{tikzcd} [row sep = tiny]
    3B \ar[r,"\kappa_1"] & A + 2B \\
    2A \ar[r,"\kappa_2"] & A + B
    \end{tikzcd}
    \begin{align*}
    f_{A} & = \kappa_{1}x_{B}^{3} -\kappa_{2}x_{A}^{2} \\
    f_{B} & = -\kappa_{1}x_{B}^{3} + \kappa_{2}x_{A}^{2}
    \end{align*}
    \end{subfigure}
    \begin{subfigure}{0.3\textwidth}\ContinuedFloat
    \centering
    \includegraphics[scale=0.37]{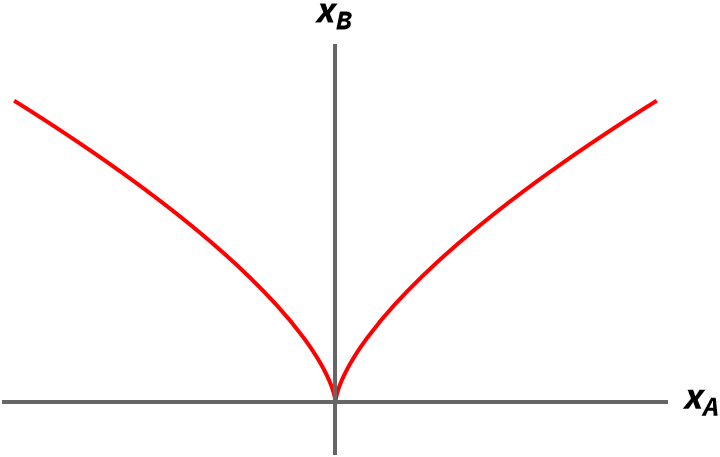}
    \end{subfigure}
    \caption{Semicubical parabola positive steady-state variety}
    \label{fig:semicubic_ex}
    \begin{subfigure}{0.3\textwidth}
    \text{ }
\end{subfigure}
\end{figure}
\end{example}

The second additional positive steady-state variety class is a cubic, or any equation of the form $y=cx^3$. The following theorem classifies all at-most-trimolecular small networks which produce each type of variety.

\begin{theorem}\label{thm:trimolec}
    Given an at-most-trimolecular small network satisfying the assumptions of Theorem \ref{thm:nonempty_pssv}, the positive steady-state variety is: 
    \begin{itemize}
        \item a semicubical parabola if and only if 1) One reactant complex is bimolecular and the other is trimolecular, and 2) The supports of the reactant complexes are disjoint;
        \item a cubic if and only if 1) One reactant complex is trimolecular and the other is unimolecular, and 2) The supports of the reactant complexes are disjoint.
    \end{itemize}
\end{theorem}

\begin{proof}

($\Rightarrow$) Suppose we have an arbitrary at-most-trimolecular small network whose positive steady-state variety is a semicubical parabola or a cubic. Up to relabeling species, the simplified positive steady-state varieties are $V(x^3-cy^2)$ and $V(y-cx^3)$ respectively, for some positive $c\in\mathbb{R}$. Following similar logic to Theorem \ref{thm:bimolec}, the steady-state polynomials take on the following forms up to multiplication by constants:
$$f_i=x^3-cy^2\text{ for semicubical parabolas;}\hspace{0.5cm} f_i=y-cx^3\text{ for cubics.}$$
As with the four at-most-bimolecular variety classes, the variable components of the monomial terms give that for a semicubical parabola the reactant complexes are $2B$ and $3A$; with arbitrary labeling of species, these are a bimolecular and a trimolecular complex with disjoint supports, proving both conditions. Similarly, for a cubic the corresponding reactant complexes are $3A$ and $B$; with arbitrary labeling of species these are a trimolecular and unimolecular complex with disjoint supports, again proving both conditions.

($\Leftarrow$) 
As in Theorem \ref{thm:bimolec}, we need only consider one steady-state equation to understand the steady-state variety, and consider without loss of generality a nonzero $f_A$. 

First suppose the network has one bimolecular and one trimolecular reactant complex with disjoint supports, and without loss of generality suppose the reactant complexes are $3A$ and $2B$. Then,
$$f_A=\kappa_1(\alpha_1-3)x^3+\kappa_2(\alpha_3)y^2 \hspace{0.3cm} \Longrightarrow \hspace{0.3cm} x^3=-\frac{\kappa_2(\alpha_3)}{\kappa_1(\alpha_1-3)}y^2$$
The stoichiometric coefficients have opposite signs by assumption meaning the coefficient on $y^2$ is positive, and the equation defines a semicubical parabola crossing into the first quadrant. The positive steady-state variety is therefore a semicubical parabola.

Next, Suppose that the reactant complexes have disjoint supports with one trimolecular and the other unimolecular, and without loss of generality suppose that the reactant complexes are $A$ and $3B$. Then,
$$f_A=\kappa_1(\alpha_1-1)x+\kappa_2(\alpha_3)y^3 \hspace{0.3cm} \Longrightarrow \hspace{0.3cm} x=-\frac{\kappa_2(\alpha_3)}{\kappa_1(\alpha_1-1)}y^3$$
The coefficient on $y^3$ is positive by assumption and the equation defines a cubic crossing into the first quadrant; the positive steady-state variety is a cubic.
\end{proof}

While the degree-three varieties are new in the at-most-trimolecular case, the other four classes of positive steady-state variety discussed above can certainly be produced here as well. However, the additional flexibility provided by having equations with degree at most three complicates the necessary conditions and increases the number of possible cases. The relevant theorems as stated in Section \ref{sec:bimolec} are therefore not sufficient for the at-most-trimolecular case; their potential generalization is discussed in Section \ref{sec:discussion}.

\begin{remark} \label{rmk:product_complexes}
    Since the classification theorems in Sections \ref{sec:bimolec} and \ref{sec:semicubical} only directly impose conditions upon the reactant complexes, they also classify the positive steady-state varieties of networks with product complexes of higher molecularities, provided that they satisfy the assumptions of Theorem \ref{thm:nonempty_pssv}. For example, the positive steady-state variety of the network comprised of $2A\rightarrow 3A+B$ and $A+B\rightarrow 0$ is a line through the origin by Theorem \ref{thm:bimolec}. 
\end{remark}

\section{Reactant Complex Overlap} \label{sec:overlap}

The theorems in Sections \ref{sec:bimolec} and \ref{sec:semicubical} rely heavily on the molecularity bounds on the complexes to classify the positive steady-state varieties. While most reactions appearing in nature are indeed bimolecular, the mathematical study of reaction networks is not limited by the same constraints. With the goal of generalizing the 2-species classification theorems to small networks with complexes of higher molecularity, we return to some more general properties of networks.
A direct result of the assumption of mass action kinetics is the connection between a network's reactant complexes and its corresponding steady-state equations' monomial terms; this relationship is fundamental to all previous proofs in this paper. Consequently, shared species among a network's reactant complexes correspond to coordinate hyperplane components of the variety. The species shared by a set of complexes are the intersection of their \textit{multiset supports}, where the multiset support of a complex is a list of all copies of each species appearing in a complex. For example, the multiset support of $2A+B$ is $\{A,A,B\}$ while its support is simply $\{A,B\}$.

\begin{prop} \label{prop:reactant_overlap}
    A network of any size will have shared species among its reactant complexes if and only if its variety has coordinate hyperplane components whose order corresponds to the multiplicity of the shared corresponding species.
\end{prop}

\begin{proof}    
Suppose all reactant complexes of a network with $s$ species and $n$ reactions share $\alpha_{i}A_i$, so that $\alpha_{ij} \geq \alpha_{i} > 0$ for all $j$. For the $m^{\text{th}}$ term of the $j^{\text{th}}$ steady-state polynomial, the stoichiometric coefficient will be  $(\alpha_{j(2m-1)}-\alpha_{j(2m-2)})$ and the exponent on each $x_k$ will be $\alpha_{k(2m-2)}$. 
Combining this with the general form steady-state polynomials yields the following after factoring for any $A_j$:
\begin{align*} 
f_{A_j} =  x_i^{\alpha_{i}}\big[(\alpha_{j1} - \alpha_{j0})\kappa_1 x_1^{\alpha_{10}}&\cdots x_{i}^{\alpha_{i0}-\alpha_i} \cdots x_s^{\alpha_{s0}} + \dotsm\\ 
& \dotsm + (\alpha_{j(2n-1)} - \alpha_{j(2n-2)})\kappa_n x_1^{\alpha_{1(2n-2)}}\cdots x_{i}^{\alpha_{i(2n-2)}-\alpha_{i}}\cdots x_s^{\alpha_{s(2n-2)}} \big]\\
\end{align*}
\noindent Defining the $n$-term polynomial within the brackets for $f_{A_j}$ as $P_j(x)$, the steady-state variety then becomes
$$\left[V(x^{\alpha_{i}}) \cup V(P_1(x))\right] \cap\cdots\cap \left[V(x^{\alpha_{i}}) \cup V(P_s(x))\right]= V(x^{\alpha_{i}}) \cup \left[V(P_1(x)) \cap\cdots\cap V(P_s(x))\right]$$
This variety is composed of a coordinate hyperplane solution when $x_i=0$ of order $\alpha_{i}$, along with the intersection of the zero sets of the remaining polynomials.

Now, suppose a network has a coordinate hyperplane of order $m$ as a component its steady-state variety, and suppose the component is $V(x_i^m)$. Then, $x_i^m=0$ must be a solution to the steady-state system, and therefore must be a shared factor among all monomial terms. Each monomial term is built from the reactant complexes, and so every reactant complex must contain $mA_i$.
\end{proof}

\begin{example}

Consider the network in Figure \ref{fig:overlap} along with its steady-state polynomials and variety, here with $\kappa_1=\kappa_2=1$ and the $x_B$ axis as the vertical axis. The shared species of the reactant complexes, here $\{A,A,B\}$, result in both axes as components of the steady-state variety (shown in red). The $x_B$ axis has multiplicity two due to the two shared copies of $A$.

\begin{figure}[ht]
    \begin{subfigure}{0.06\textwidth}
    \text{ }
    \end{subfigure}
    \begin{subfigure}{0.55\textwidth}
    \centering
    \begin{tikzcd}
     3A+B \ar[r,"\kappa_1", shift left] & 2A+3B \ar[l, "\kappa_2", shift left]
    \end{tikzcd}
    \medskip
    \begin{align*}
    f_{A} & = -\kappa_{1}x_{A}^{3}x_B +\kappa_{2}x_A^2x_{B}^{3}=x_A^2x_B(-\kappa_1x_A+\kappa_2x_B^2) \\
    f_{B} & = \kappa_{1}x_{A}^{3}x_B - \kappa_{2}x_A^2x_{B}^{3}=x_A^2x_B(\kappa_1x_A-\kappa_2x_B^2)
    \end{align*}
    \medskip
    \end{subfigure}
    \begin{subfigure}{0.3\textwidth}
    \centering
    \includegraphics[scale=0.33]{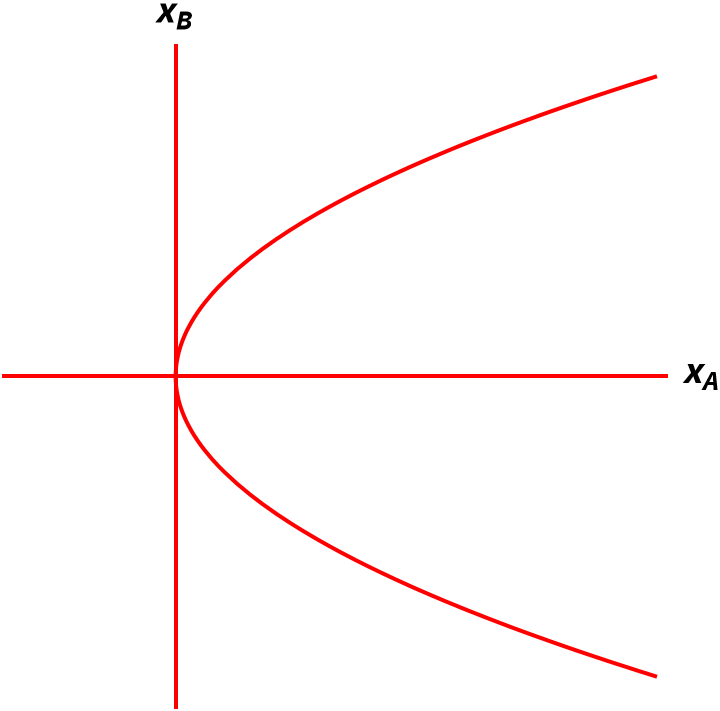}
    \end{subfigure}
    \begin{subfigure}{0.06\textwidth}
    \text{ }
    \end{subfigure}
    \caption{Reactant complex overlap and axis components}
    \label{fig:overlap}

\end{figure}
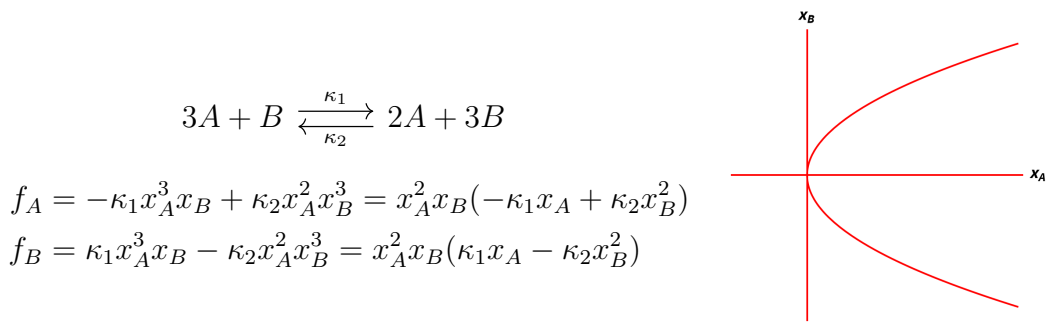
\end{example}

It follows immediately that two distinct networks have the potential to produce steady-state varieties whose non-coordinate-hyperplane components are identical if the non-shared components of their reactant complexes are the same; however, differences in the product complexes could alter the variety by changing the stoichiometric coefficients. To identify some cases where distinct networks can produce the same variety, we consider actions under which a network's non-coordinate-hyperplane steady-state variety is invariant.
Observe that the network in Figure \ref{fig:overlap} has the same non-axis steady-state variety as the example network in Section \ref{sec:CRN}, and is indeed the same network with $2A+B$ added to both reactant complexes. The following property allows the construction of higher-molecularity networks within a given variety family by building off of a smaller network.

\begin{corollary}\label{cor:extensions}
A network's non-coordinate-hyperplane steady-state variety is invariant under the operation of adding a linear combination of species to all complexes. 
\end{corollary}

\begin{proof}
Consider the transformation of adding $m_iA_i$ to all complexes of a network, where the $m_i$ are nonnegative integers and the following is the $i^{\text{th}}$ reaction:
$$\begin{tikzcd} [row sep=tiny]
(\alpha_{1(2i-2)} + m_1) A_1 +\dotsm + (\alpha_{s(2i-2)} + m_s) A_s \ar[r, "\kappa_i"] & (\alpha_{1(2i-1)} + m_1) A_1 +\cdots + (\alpha_{s(2i-1)} + m_s) A_s\\
\end{tikzcd}$$
The corresponding steady-state polynomial for each $A_i$ after simplification is
$$f_{A_i}=x_1^{m_1}\cdots x_s^{m_s}\left[ (\alpha_{i1}-\alpha_{i0})\kappa_1 x_1^{\alpha_{10}}\cdots x_s^{\alpha_{s0}}+\cdots+(\alpha_{i(2n-1)}-\alpha_{i(2n-2)})\kappa_n x_1^{\alpha_{1(2n-2)}}\cdots x_s^{\alpha_{s(2n-2)}} \right]$$
The added species are shared by the reactant complexes, so by Proposition \ref{prop:reactant_overlap} these steady-state equations formulate the same non-axis variety as the network before alteration, as the additional species only add coordinate hyperplane solutions and the $P_j(x)$ are unaltered.
\end{proof}

\begin{remark}
The operation described in Corollary \ref{cor:extensions} is very similar to the network ``translation'' operation studied in    \cite[\S 3]{MST2021ACR}.
\end{remark}

\begin{remark}\label{rmk:extension} Using the ideas introduced above, we can consider the \textbf{reduced reactant complexes} of a network, which are the reactant complexes with all shared species removed. In terms of multiset supports, the reduced reactant complexes are the reactant complexes minus the intersection of their multiset supports. By Proposition \ref{prop:reactant_overlap}, the reduced reactant complexes contain all necessary information about the non-coordinate-hyperplane components of the steady-state variety. So, the theorems in Sections \ref{sec:bimolec} and \ref{sec:semicubical} can be extended to networks with complexes of higher molecularity in a limited sense. If the reduced reactant complexes of a 2-species, 2-reaction network meets the conditions of any of the classification theorems, then it will produce the indicated class of variety. However, the converse is not true in general, discussed below.
\end{remark}

\section{Discussion}\label{sec:discussion}

As mentioned in the introduction, calculating the positive steady-state variety is difficult in general and existing tools are limited in their scope. With Theorem \ref{thm:nonempty_pssv}, the case of 2-reaction networks is fully described. This paper also classified 2-species at-most-bimolecular networks by the shape of their positive steady-state variety (Section \ref{sec:bimolec}). The reactant complex conditions are verifiable by inspection and the construction of the stoichiometric matrix is straightforward, providing a significant simplification of the conditions given in \cite{MDSC2012toric} to identify networks with nonempty toric positive steady-state varieties. Finally, we gave some potential tools to expand the classification theorems of Sections \ref{sec:bimolec} and \ref{sec:semicubical} to networks with complexes of higher molecularity (Section \ref{sec:overlap}). 

Attempting to extend the above classification theorems to higher-molecularity networks raises some interesting questions. As mentioned in Remark \ref{rmk:extension}, the classification theorems in this paper give a partial description of higher-molecularity networks when considering the reduced reactant complexes; however, a complete classification would require a different approach. As an example, consider a parabolic steady-state variety. In the at-most-bimolecular case, this class of variety can only be generated by $y-cx^2$ as seen in the proof of Theorem \ref{thm:bimolec}. Following similar logic and using the framework of reduced reactant complexes introduced in Remark \ref{rmk:extension}, we could extend this theorem and conclude that, if a network's reduced reactant complexes meet the parabola conditions of Theorem \ref{thm:bimolec}, then its positive steady-state variety is a parabola. However, the converse would not hold. Consider for example the network comprised of $4A\rightarrow 3A+B$ and $2B\rightarrow A+B$, which satisfies the conditions of Theorem \ref{thm:nonempty_pssv} but not the parabola conditions of Theorem \ref{thm:bimolec}. Its steady-state variety is defined by $y^2=\kappa_1/\kappa_2\cdot x^4$, or $y=\pm\sqrt{\kappa_1/\kappa_2} \cdot x^2$. The positive steady-state variety is therefore a parabola. While the (reduced) reactant complexes do not meet the criteria of the parabola classification theorem, they do share some similarities. The supports are disjoint, and the molecularity of one is twice that of the other. This suggests that a complete classification of higher-degree positive steady-state varieties based in similar network properties would require considering the ratio between and parities of the reduced reactant complexes' molecularities. 

While these generalizations are mathematically interesting, they are less applicable to real-world reaction networks than the at-most-bimolecular case. In chemistry, most valid reactions involve bimolecular reactant complexes (and, rarely, trimolecular ones), since the likelihood of more than two components coming into contact simultaneously, in the proper orientation and with sufficient energy, is very low. So, generalizations of our classification theorems to networks with complexes of higher molecularities would describe few biologically relevant networks. However, as mentioned in Remark \ref{rmk:product_complexes}, the theorems of Sections \ref{sec:bimolec} and \ref{sec:semicubical} apply to 2-reaction networks with product complexes of higher molecularity, provided that the reactant complexes are bimolecular. The authors of \cite{BBHosc2024} studied 3-reaction so-called ``bimolecular-sourced" networks and their oscillatory behavior. Generalizing the systematic approach and simplified network conditions of this paper to larger networks---for example, to 3-reaction networks---could lend valuable insight into biologically relevant bimolecular-sourced networks by providing a more holistic understanding of their positive steady-state varieties.

\section*{Acknowledgements}

This research was completed as part of the 2023 Pomona Research in Mathematics Experience, supported by Pomona College and the NSF (DMS-1560394). We thank Edray Goins, Alex Barrios, and all other faculty and participants of PRiME for their guidance, collaboration, and support. We also thank Elizabeth Gross, Eddie Price, and Anne Shiu for their comments at various stages of this project.


\bibliographystyle{plain}
\bibliography{ref}

\end{document}